\documentclass[a4paper,12pt,reqno]{amsart}
\usepackage[english]{babel}
\usepackage[a4paper, margin=2.5cm]{geometry}
\usepackage{amsmath,amssymb,mathtools,amsthm,tikz-cd,hyperref,graphics,csquotes}
\usepackage[shortlabels]{enumitem}

\title{Noetherian Rings of non-local Rank}
\author{Dmitry Kudryakov}
\address{Saint Petersburg University, 7/9 Universitetskaya nab., St. Petersburg, 199034, Russia}
\email{d.a.kudryakov@gmail.com}
\thanks{Preprint available at \href{https://arxiv.org/abs/2501.05940v2}{arXiv:2501.05940v2} [math.AC]}

\makeatletter
\@namedef{subjclassname@2020}{\textup{2020} Mathematics Subject Classification}
\makeatother

\newtheorem{theorem}{Theorem}[section]

\newtheorem{lemma}[theorem]{Lemma}
\newtheorem{corollary}[theorem]{Corollary}
\newtheorem{proposition}[theorem]{Proposition}
\theoremstyle{definition}

\newcommand{\ti}{\times}

\newcommand{\sm}{\setminus}

\renewcommand{\ss}{\subseteq}
\newcommand{\m}{\mathfrak{m}}

\newcommand{\p}{\mathfrak{p}}
\newcommand{\cn}{\colon}

\newcommand{\n}{\mathfrak{n}}
\newcommand{\q}{\mathfrak{q}}

\newcommand{\Sup}{\sup\limits}
\newcommand\Prod{\prod\limits}

\newcommand{\rk}{\operatorname{rk}}

\let\le\leqslant
\let\ge\geqslant

\newcommand{\brc}[1]{\!\left(#1\right)\!}
\newcommand{\brs}[1]{\!\left\{#1\right\}\!}

\newcommand{\Spec}[1]{\operatorname{Spec}#1}
\newcommand{\Max}{\operatorname{MaxSpec}}
\newcommand{\Min}{\operatorname{MinSpec}}
\newcommand{\nil}{\operatorname{nil}}
\newcommand{\lrk}{\operatorname{lrk}}
\newcommand{\hth}{\operatorname{ht}}

\begin{document}

\date{\today}



\begin{abstract}
    The rank of a ring $R$ is the supremum of minimal cardinalities of generating sets of $I$, among all ideals $I$ in $R$.
    In this paper, we obtain a characterization of Noetherian rings $R$ whose rank is not equal to the supremum of ranks of localizations of $R$ at maximal ideals.
    It turns out that any such ring is a direct product of a finite number of local principal Artinian rings and Dedekind domains, at least one of which is not a principal ideal ring.
    As an application, we show that the rank of the ring of polynomials over an Artinian ring can be computed locally.
\end{abstract}

\maketitle

\section{Introduction}

Let $R$ be a commutative Noetherian ring.
For a cardinal number $n$, the condition $\rk R=n$ means that any ideal in the ring $R$ is generated by a set of cardinality $n$, but there exists an ideal that is not generated by a set of cardinality less than $n$.
Let us define
$
\lrk R\coloneqq \Sup_{\m\in\Max R}\rk R_\m.
$
In \cite{Cohen-1950} Cohen showed that for an integral domain $R$ we have $\rk R\le \lrk R+1$.
Using ideas of Forster \cite{Forster} and Swan \cite{Swan}, it is not difficult to show that the result of Cohen generalizes to arbitrary rings, i.e.
$$
\lrk R\le\rk R\le\lrk R+1,\qquad\text{for any Noetherian ring }R.
$$
(For details see Proposition \ref{upper bound for rank}.)
Let us say that $R$ is \emph{of local rank} if $\rk R=\lrk R$.
Hence one can pose a natural question: which rings are not of local rank?

Rings of dimension zero are of local rank, since semi-local rings are of local rank, due to \cite{Forster,Swan}.
For rings of dimension at least two
the answer is the same, due to the characterization of local rings of finite rank from \cite{Sally-1978}.
(For details see Propositions \ref{Artinian is locally ranked} and \ref{dimension 2 is not finite rank}.)
It remains to treat the case of dimension one.

It is well known that the rank of a non-principal Dedekind domain $R$ is two, which is greater than the rank of the localization of $R$ at any maximal ideal.
Hence non-principal Dedekind domains are not of local rank.
The following result of Clark shows that a non-Dedekind domain is of local rank.
\begin{theorem}[{\cite[Theorem 2.4]{Clark-2017}}]\label{motivation}
    Let $R$ be a one-dimensional non-normal domain. Then $R$ is of local rank.
\end{theorem}
This result is the motivation to this paper.
The purpose of this paper is to describe one dimensional rings of local rank which are not domains.
The main result of this paper is the following.

\begin{theorem}\label{main result}
    For a Noetherian ring $R$, the following statements are equivalent.
    \begin{enumerate}
        \item\label{first} The ring $R$ is not of local rank, i.e. $\rk R=\lrk R+1$.
        \item\label{decomp} The ring $R$ is a direct product of finite number of local principal ideal Artinian rings and Dedekind domains with at least one non-principal ideal ring.
    \end{enumerate}
\end{theorem}
We prove this Theorem in Section \ref{proof of main result}.
Moreover, we prove that the decomposition of a ring of non-local rank into a direct product of the form (\ref{decomp}) is uniquely determined up to isomorphism by $\Min R$ (Theorem \ref{main lemma}).
As an auxiliary result, using ideas \cite[Proof of Theorem 3.1.1]{Sally-1978} from Sally's book, we show that the rank of a local ring of dimension one cannot be realized on an ideal of height zero (Lemma \ref{ideal of height 0}).

In Subsection \ref{reduced rings} we generalize Theorem \ref{motivation} to the case of reduced rings.
Clark showed \cite[Theorem 4.4]{Clark-2017} that the rank of a ring of polynomials $A[t]$ is finite if and only if $A$ is an Artinian ring.
It is also shown in that paper that in this case the rank of $A[t]$ does not exceed the length of the ring $A$ \cite[Theorem 4.5]{Clark-2017}.
As an application of Theorem \ref{main result}, in Subsection \ref{polynomial rings} we show that $A[t]$ is of local rank, i.e.
$$
\rk A[t]=\lrk A[t],\qquad\text{for any Artinian ring }A.
$$

\subsection*{Notation}
All rings are assumed to be associative, commutative, unitial and Noetherian.
By $\Min R$ and $\Max R$ we denote sets of all minimal prime ideals and, respectively, of all maximal ideals of a ring $R$.
Since $R$ is Noetherian, $\Min R$ is a finite set.
For a module $M$ over a ring $R$, we denote by $\mu_R(M)=\mu(M)$ the minimal number of generators of $M$ as an $R$-module.
Note that $\rk R=\sup\brs{\mu\brc{I}\mid I\ss R}$.

\subsection*{Preliminaries}
We start with the basic properties of functions $\mu$ and $\rk$.

\begin{proposition}[{\cite[Remark 1.2]{Clark-2017}}]\label{main prop on mu}
    Let $f\cn R\to A$ be a ring homomorphism, $M$ be an $R$-module and $f_*(M)$ be the $A$-module $M\otimes_RA$. Then
    \begin{enumerate}
        \item\label{generators of localization} $\mu(f_*(M))\le\mu(M)$; for example, this holds if $f$ is a quotient or localization map;
        \item if $I$ is an ideal in $R$, then $\rk(R/I)\le\rk(R)$;
        \item\label{rank of localization} for any mutiplicative set $S$ in $R$ we have $\rk(S^{-1}R)\le\rk(R)$.
    \end{enumerate}
\end{proposition}

\begin{proposition}[{\cite[Theorem 4.1]{Clark-2017}}]\label{rank of finite direct product}
    Let $R=\Prod_{i=1}^nR_i$ be a direct product of finite number of rings. Then
    $
    \rk R=\max\limits_{i\in\brs{1,\ldots,n}} \rk R_i.
    $
\end{proposition}

We will need the following corollary of \cite[Theorem 2]{Swan}.

\begin{theorem}[\cite{Swan}]\label{cor from Swan}
    For a ring $R$ we have $\rk R\le\Sup_{\p\in\Spec R}\brc{\rk R_\p+\dim R/\p}$.
\end{theorem}

\begin{proposition}[{\cite[Chapter 1, 2.1]{Sally-1978}}]\label{multiplicity of one dimensional ring}
    Let $R$ be a local ring with maximal ideal $\m$ and $\dim R=1$. Then $e(R)=\dim_{R/\m}\m^n/\m^{n+1}$ is a positive constant for large enough $n$ and is called multiplicity of $R$.
\end{proposition}

\begin{theorem}[{\cite[Chapter 3, Theorem 1.2]{Sally-1978}}]\label{rings of dim at most 1}
    If $R$ is local, then $\dim R\le 1$ if and only if $\rk R<\infty$.
\end{theorem}

The following statement relates the rank of a ring to its local rank.
Its proof depends essentially on Theorem \ref{cor from Swan}.

\begin{proposition}\label{upper bound for rank}
    For a ring $R$ it follows that
    $
    \lrk R\le\rk R\le\lrk R+1.
    $ In particular, $R$ is of local rank if $\rk R$ is infinite.
\end{proposition}

\begin{proof}
    The first inequality follows from \textit{(\ref{rank of localization})}.
    For the second one,
    assume firstly that $\dim R\ge 2$.
    Then $\dim R_\n\ge 2$ for some $\n\in\Max R$.
    Since $R$ is Noetherian, it follows from Theorem \ref{rings of dim at most 1} that $\rk R_\n$ is countably infinite.
    Again since $R$ is Noetherian, from the first inequality we see that $\rk R$ is also countably infinite.
    Hence the second inequality holds and is an equality in fact.
    Now assume that $\dim R\le 1$.
    Thus
    \[
    \rk R\stackrel{\ref{cor from Swan}}{\le}\Sup_{\p\in\Spec R}\brc{\rk R_\p+\dim R/\p}\stackrel{\textit{(\ref{rank of localization})}}{\le}\Sup_{\m\in\Max R}\brc{\rk R_\m+1}\le\lrk R+1.\qedhere
    \]
\end{proof}

\begin{proposition}\label{Artinian is locally ranked}
    If $\dim R=0$, then $\rk R<\infty$ and $R$ is of local rank.
\end{proposition}

\begin{proof}
    Since the Artinian ring $R$ is a direct product of finite number of local Artinian rings, Proposition \ref{rank of finite direct product} and Theorem \ref{rings of dim at most 1} imply that $R$ is of finite rank.
    From $\dim R=0$ it follows that $\Spec R=\Max R$ and $\dim R/\p=0$ for any prime $\p\ss R$.
    Hence Theorem \ref{cor from Swan} implies $\rk R\le\lrk R$.
    Thus the result follows from Proposition \ref{upper bound for rank}.
\end{proof}

\begin{proposition}\label{dimension 2 is not finite rank}
    If $\dim R\ge 2$, then $\rk R=\infty$ and $R$ is of local rank.
\end{proposition}

\begin{proof}
    Assume $\rk R=n<\infty$.
    Hence Proposition \ref{main prop on mu} implies $\rk R_\m\le n$ for any\linebreak$\m\in\Max R$. By Theorem \ref{rings of dim at most 1}, for any such $\m$ we have $\dim R_\m\le 1$. It follows that $\dim R\le 1$, a contradiction.
    Thus from Proposition \ref{upper bound for rank} it follows that $R$ is of local rank.
\end{proof}

\section{The Main Theorem}\label{proof of main result}

In this section we prove Theorem \ref{main result}.
It remains to treat the case of dimension $1$, which is the hardest one.
Remind that $\hth I=\min\brs{\hth \p\mid I\ss\p\in\Spec R}$ is the height of an ideal $I\ss R$.

\begin{lemma}\label{ideal of height 0}
    Let $I$ be an ideal in a local ring $R$ with $\dim R=1$.
    If $\hth I=0$, then $\mu(I)<\rk R$.
\end{lemma}

\begin{proof}
    Let us consider $\m\in\Max R$.
    By Artin-Rees Lemma applied to $\m$-adic filtration of $R$, there is a positive integer $n_0$ such that $\m I\supseteq\m\brc{I\cap\m^{n-1}}=I\cap\m^{n}$ for all $n\ge n_0$.
    Therefore, for all $n\ge n_0$ there is a natural isomorphism $I/\m I\cong \brc{I+\m^n}/\m\brc{I+\m^{n-1}}$.
    
    The exact sequence
    $$
    0\to \brc{I+\m^{n}}/\m\brc{I+\m^{n-1}}\to\brc{I+\m^{n-1}}/\m\brc{I+\m^{n-1}}\to\brc{I+\m^{n-1}}/\brc{I+\m^{n}}\to 0
    $$
    of $R/\m$-modules and Nakayama's Lemma give
    \begin{align*}
        \rk R\ge \mu\brc{I+\m^n}&=\dim_{R/\m}\brc{I+\m^{n-1}}/\m\brc{I+\m^{n-1}}
        \\
        &=\dim_{R/\m}\brc{I+\m^{n}}/\m\brc{I+\m^{n-1}}+\dim_{R/\m} \brc{I+\m^{n-1}}/\brc{I+\m^{n}}.
    \end{align*}

    Thus,
    since $\dim R/I=1$, Proposition \ref{multiplicity of one dimensional ring} implies
    $$
    \rk R\ge\dim_{R/\m}I/I\m+e\brc{R/I}=\mu\brc{I}+e\brc{R/I},\qquad e\brc{R/I}\ge 1,
    $$
    for $n$ large enough.
    But $\rk R<\infty$, by Theorem \ref{rings of dim at most 1}. So the statement follows.
\end{proof}

\begin{lemma}\label{1 to 2}
    Let $R$ be a ring of non-local rank. Then $R$ has the following properties:
        \begin{enumerate}
            \item $\dim R=1$,
            \item $R_\p$ are fields for all $\p\in\Min R\sm\Max R$,
            \item $R_\m$ are DVR's for all $\m\in\Max R\sm\Min R$,
            \item $R_\m$ are principal ideal Artinian rings for all $\m\in\Min R\cap\Max R$,
            \item $\rk R=2$.
        \end{enumerate}
\end{lemma}

\begin{proof}
    From Corollary \ref{Artinian is locally ranked} and Proposition \ref{dimension 2 is not finite rank} we have $\dim R=1$.
    
    Notice that if $\rk R=\infty$, then $R$ is of local rank by Proposition \ref{upper bound for rank}, which contradicts the assumption.
    Hence $\rk R<\infty$, and Corollary \ref{cor from Swan} implies $\rk R\le \rk R_\q+\dim R/\q$, for some prime $\q\ss R$.
    Since $R$ is of non-local rank, Proposition \ref{main prop on mu} implies $\q\not\in\Max R$ and $\rk R_\q=\rk R-1$. 
    Consider some $\n\in\Max R$ with $\q\ss\n$. 
    From Lemma \ref{ideal of height 0} we see that $\mu\brc{I}<\rk R_\n$ for any ideal $I\ss\q R_\n$.
    Thus if we denote by $\lambda$ the localization map $R_\n\to R_\q$, then for any ideal $J\ss \q R_\q$ Proposition \ref{main prop on mu} gives
    \begin{equation}\label{*}
        \mu_{R_\q}\brc{J}\le\mu_{R_\n}\brc{\lambda^{-1}\brc{J}}<\rk R_\n\le\rk R-1=\rk R_\q.
    \end{equation}
    Since $\rk R_\q<\infty$, we see that $\rk R_\q=\mu_{R_\q}(J_0)$ for some ideal $J_0\ss R_\q$.
    And since $R_\q$ is local, it follows from \eqref{*} that any such $J_0$ should satisfy $J_0=R_\q$.
    Hence $\rk R=2$ and $\rk R_\p=1$ for all $\p\in\Spec R$.
    
    To sum up, if $\p\in\Min R$, then $R_\p$ is a principal ideal Artinian ring; if\linebreak$\p\in\Max R\sm\Min R$, then $R_\p$ is a regular local ring with $\dim R_\p=1$, hence is a DVR. Finally, for $\p\in\Min R\sm\Max R$ and $\p\ss\m\in\Max R\sm\Min R$ we conclude that $R_\p$ is the quotient field of $R_\m$.
\end{proof}

\begin{theorem}\label{main lemma}
    Let $R$ be a ring of non-local rank.
    Then $R$ decomposes uniquely (up to isomorphism) into a direct product of finite number of local principal ideal Artinian rings and Dedekind domains with at least one non-principal ideal ring.
    The Artinian components are $R_\p$ with $\p\in\Min R\cap\Max R$, while the Dedekind components are $R/\q$ with $\q\in\Min R\sm\Max R$.
\end{theorem}

\begin{proof}
    Let $\p_1,\ldots,\p_n$ be all minimal primes in $R$.
    First note that any $\p_i,\p_j$ are coprime, for otherwise $\p_i,\p_j\ss\m$ for some maximal $\m\ne\p_i,\p_j$, hence from Lemma \ref{1 to 2} we see that in the DVR $R_\m$ there are two nonmaximal primes, a contradiction.
    Since $R$ is Noetherian, there is a number $N$ with the property $\p_1^N\ldots\p_n^N\ss\brc{\nil R}^N=0$.
    Thus the Chinese remainder theorem implies $R\cong R/\p_1^N\ti\ldots\ti R/\p_n^N$.

    Let us consider $\p=\p_i\in\Min R\cap\Max R$.
    Hence Lemma \ref{1 to 2} implies that the rings $R/\p^N$ and $R_\p$ are local Artinian.
    Since $\p^NR_\p=\p_1^N\ldots\p_n^NR_\p=0$, it follows that $R/\p^N\cong\brc{R/\p^N}_\p\cong R_\p/\p^NR_\p\cong R_\p$.

    Now assume $\q=\p_i\in\Min R\sm\Max R$.
    For any $\m\in\Max R$ we get $\q R_\m=\q^2 R_\m=R_\m$ if $\q\not\ss\m$.
    If $\q\ss\m$, then Lemma \ref{1 to 2} gives that $R_\m$ is a domain, so $\q R_\m=\q^2 R_\m=0$.
    Thus $\q$ and $\q^2$ agree locally, hence $\q=\q^2$ and $R/\q^N=R/\q$ is a domain of dimension $1$.
    However, Lemma \ref{1 to 2} says that $\brc{R/\q}_\m\cong R_\m$ is a DVR for any $\m\supseteq\q$.
    It follows that $R/\q$ is a Dedekind domain.
    From Lemma \ref{1 to 2} we have $\rk R=2$, so at least one of the Dedekind components is non-principal (by Proposition \ref{rank of finite direct product}).
\end{proof}

\begin{proof}[Proof of Theorem \ref{main result}]
    Statement (\ref{decomp}) follows from (\ref{first}) by Theorem \ref{main lemma}.
    
    Conversely, assume (\ref{decomp}) holds.
    Hence Proposition \ref{rank of finite direct product} implies $\rk R= 2$.
    However, all maximal localizations of $R$ are localizations of either principal ideal rings or Dedekind domains, hence are principal. Thus
    $\lrk R=1<\rk R$,
    so $R$ is not of local rank.
\end{proof}

\section{Examples}\label{examples}

\subsection{Reduced rings}\label{reduced rings}

From Theorem \ref{main result} we obtain the following result.

\begin{corollary}\label{main for reduced case}
    A reduced ring is of local rank if and only if it is not a direct product of finite number of fields and Dedekind domains with at least one non-principal ideal ring.
\end{corollary}

Hence we have the following generalization of Theorem \ref{motivation} to an arbitrary dimension.

\begin{corollary}
    An integral domain is of local rank if and only if it is not a non-principal Dedekind domain.
\end{corollary}

\subsection{Polynomial rings}\label{polynomial rings}

Let $R$ be a Noetherian polynomial ring.
This means that\linebreak$R=A[t_1,\ldots,t_n]$, where $A$ is a ring, and $t_1,\ldots,t_n$ are indeterminates.
If $n\ge 2$, then $\dim R\ge 2$ and from Theorem \ref{main result} one has that $R$ is of local rank.
So we may assume $n=1$.
As $A$ is a quotient of $R$, it is Noetherian.
If $\dim A\ge 1$, then $\dim R\ge 2$, so again by Theorem \ref{main result} $R$ is of local rank.
Hence, assume now that $A$ is Artinian.
It is known that $R=A[t]$ is of finite rank \cite[Theorem 4.4]{Clark-2017}.

\begin{proposition}\label{polynomials over ring is can be computed locally}
    The ring $R=A[t]$ is of local rank.
\end{proposition}

\begin{proof}
    Suppose $R$ is not of local rank.
    Hence Theorem \ref{main result} implies that $B\coloneqq R/\nil R$ is a direct product of finite number of fields and Dedekind domains with at least one non-principal ideal ring.
    Thus $\rk B=2$, by Proposition \ref{rank of finite direct product}.
    
    On the other hand, $\dim A=0$ implies $A/\nil A\cong F_1\ti\ldots\ti F_n$ for some fields $F_i$.
    By Proposition \ref{rank of finite direct product}, $B\cong \brc{A/\nil A}[t]\cong F_1[t]\ti\ldots\ti F_n[t]$ is a principal ideal ring, which contradicts $\rk B=2$.
\end{proof}

\section*{Acknowledgements}

I would like to thank Alexander Sivatski for the proof of the uniqueness in Theorem \ref{main lemma}, as well as for advice and helpful comments.

This work was performed at the Saint Petersburg Leonhard Euler International Mathematical Institute and supported by the Ministry of Science and Higher Education of the Russian Federation (agreement no. 075-15-2025-343).

\end{document}